\documentclass[11pt]{article}
\usepackage{mathrsfs}
\usepackage{amsmath}
\usepackage{bm}
\usepackage{bm}
\usepackage{amsfonts}
\usepackage{amssymb}
\usepackage{color}
\usepackage{latexsym}
\usepackage{hyperref}
\usepackage[dvips]{graphicx}
\usepackage[left=1in, right=1in]{geometry}

\newtheorem{theorem}{Theorem}[section]
\newtheorem{lemma}[theorem]{Lemma}
\newtheorem{proposition}[theorem]{Proposition}
\newtheorem{corollary}[theorem]{Corollary}

\newtheorem{definition}[theorem]{Definition}
\newenvironment{proof}{\noindent{\em Proof:}}{\quad \hfill$\Box$\vspace{2ex}}

\newcommand*\mathinhead[2]{\texorpdfstring{$\boldsymbol{#1}$}{#2}}

\usepackage[left,pagewise,modulo, displaymath, mathlines]{lineno}
\newcommand*\patchAmsMathEnvironmentForLineno[1]{%
  \expandafter\let\csname old#1\expandafter\endcsname\csname #1\endcsname
  \expandafter\let\csname oldend#1\expandafter\endcsname\csname end#1\endcsname
  \renewenvironment{#1}%
     {\linenomath\csname old#1\endcsname}%
     {\csname oldend#1\endcsname\endlinenomath}}%
\newcommand*\patchBothAmsMathEnvironmentsForLineno[1]{%
  \patchAmsMathEnvironmentForLineno{#1}%
  \patchAmsMathEnvironmentForLineno{#1*}}%
\AtBeginDocument{%
\patchBothAmsMathEnvironmentsForLineno{equation}%
\patchBothAmsMathEnvironmentsForLineno{align}%
\patchBothAmsMathEnvironmentsForLineno{flalign}%
\patchBothAmsMathEnvironmentsForLineno{alignat}%
\patchBothAmsMathEnvironmentsForLineno{gather}%
\patchBothAmsMathEnvironmentsForLineno{multline}%
}

\makeatletter
\newcommand{\refcheckize}[1]{%
  \expandafter\let\csname @@\string#1\endcsname#1%
  \expandafter\DeclareRobustCommand\csname relax\string#1\endcsname[1]{%
    \csname @@\string#1\endcsname{##1}\wrtusdrf{##1}}%
  \expandafter\let\expandafter#1\csname relax\string#1\endcsname
}
\makeatother

\refcheckize{\cref}
\refcheckize{\Cref}

\def \mA {\mathsf A}

\newcommand{\vb}{\boldsymbol{b}}
\newcommand{\vc}{\boldsymbol{c}}
\newcommand{\ve}{\boldsymbol{e}}
\newcommand{\vt}{\boldsymbol{t}}
\newcommand{\vu}{\boldsymbol{u}}

\newcommand{\vy}{\boldsymbol{y}}
\newcommand{\vzero}{\boldsymbol{0}}

\def \bI {\mathbb I}
\def \bN {\mathbb N}

\def \bR {\mathbb R}

\def \bA {\mathbb A}

\def \bx {{\bf x}}
\def \by {{\boldsymbol y}}

\def \bb {{\boldsymbol b}}
\def \bc {{\boldsymbol c}}

\def \bz {{\boldsymbol z}}
\def \bp {{\boldsymbol p}}
\def \bq {{\boldsymbol q}}
\def \bb {{\boldsymbol b}}

\def \bv {{\boldsymbol v}}
\def \bu {{\boldsymbol u}}

\def \bt {{\boldsymbol t}}
\def \bK {{\bf K}}

\def \cB {{\cal B}}
\def \cI {{\cal I}}
\def \cS {{\cal S}}

\def \cH {{\cal H}}

\def \supp {\,{\rm supp}\,}
\def \diag {\,{\rm diag}\,}

\makeatletter

\newcommand{\Rmnum}[1]{\expandafter\@slowromancap\romannumeral #1@}
\makeatother

\begin{document}

\title{\bf Multi-task Learning in Vector-valued Reproducing Kernel Banach Spaces with the $\ell^1$ Norm}
\author{Rongrong Lin\thanks{School of Data and Computer Science, Sun Yat-sen University, Guangzhou, P. R. China. E-mail address: {\it linrr@mail2.sysu.edu.cn.} Supported in part by Fundamental Research Funds for the Central Universities under grant 18lgpy64.},
\quad Guohui Song\thanks{Department of Mathematics, Clarkson University, 8 Clarkson Ave, Potsdam, NY 13699, USA. E-mail address: {\it gsong@clarkson.edu. }  Supported in part by NSF DMS-1521661 and Natural Science Foundation of China under grant 11701383.},
\quad and \quad Haizhang Zhang\thanks{{\it Corresponding author}. School of Data and Computer Science, and Guangdong Province Key Laboratory of Computational Science, Sun Yat-sen University, Guangzhou, P. R. China. E-mail address: {\it zhhaizh2@mail.sysu.edu.cn}. Supported in part by Natural Science Foundation of China under grants 11571377 and 11222103.}}
\date{}
\maketitle

\begin{abstract}
Targeting at sparse multi-task learning, we consider regularization models with an $\ell^1$ penalty on the coefficients of kernel functions. In order to provide a kernel method for this model, we construct a class of vector-valued reproducing kernel Banach spaces with the $\ell^1$ norm. The notion of multi-task admissible kernels is proposed so that the constructed spaces could have desirable properties including the crucial linear representer theorem. Such kernels are related to bounded Lebesgue constants of a kernel interpolation question. We study the Lebesgue constant of multi-task kernels and provide examples of admissible kernels. Furthermore, we present numerical experiments for both synthetic data and real-world benchmark data to demonstrate the advantages of the proposed construction and regularization models.


\end{abstract}

\section{Introduction}
Reproducing kernel Banach spaces (RKBSs) and their applications have attracted a lot of attention in machine learning community \cite{Christensen12, FHY15,Fukumizu2011, Han2009, SongZhang,SZH13,XY18,Ye13,ZZ11, ZZ12}. In particular, RKBSs with the $\ell^{1}$ norm \cite{SongZhang,SZH13} have proven to be useful in promoting sparsity in single task learning. On the other hand, vector-valued function spaces \cite{ARL2012,Carmeli,Micchelli05} provide a solid foundation for many models in multi-task learning. The purpose of this paper is to construct vector-valued RKBSs with the $\ell^1$ norm and study regularization methods for multi-task learning in such spaces.

Reproducing kernel Hilbert spaces (RKHSs) are Hilbert spaces of functions on which point evaluation functionals are continuous \cite{Aronszajn}. In machine learning, RKHSs have been viewed as ideal spaces for kernel-based learning algorithms \cite{Cuckerzhou07,Scholkopf2001,Steinwart08,Wendland2005}. Thanks to the existence of an inner product, Hilbert spaces  are well-understood in functional analysis. Most importantly, an RKHS has a reproducing kernel, which measures similarity between inputs and gives birth to the ``kernel trick" in machine learning that significantly saves computations. Celebrated machine learning methods based on scalar-valued RKHSs include support vector machines and the regularization networks.

The RKBS is a recent and fast-growing research area. We mention two reasons that justify the need of RKBSs here. On one hand, Banach spaces possess richer geometrical structures and norms.  It is standard knowledge in functional analysis that any two Hilbert spaces on a common number field of the same dimension are isometrically isomorphic to each other, and hence share the same norms and geometry. By contrast,  for $1\le p\ne q\le+\infty$, $L^p([0,1])$ and $L^q([0,1])$ are not  isomorphic to each other.  On the other hand, many important problems such as $p$-norm coefficient-based regularization \cite{TCY2010}, large-margin classification \cite{Der2007,Ye13,ZXZ09}, lasso in statistics \cite{Tibshirani1996} and compressed sensing \cite{CRT2006} had better be studied in Banach spaces.

There are various approaches to constructing scalar-valued RKBSs in the literature. For example, \cite{ZXZ09} employs the tool of semi-inner-products to build RKBSs and \cite{XY18} constructs RKBSs based on certain feature mappings. In particular, a bilinear form has been used to develop RKBSs with the $\ell^1$ norm in \cite{SZH13}. Moreover, a recent work \cite{LinUnifRKBS} gives a unified definition of RKBSs that is more general than the aforementioned specific ones. It also proposed a unified framework of constructing scalar-valued RKBSs that covers all existing constructions \cite{FHY15,SongZhang,SZH13,XY18,Ye13,ZZ11, ZZ12} via a continuous bilinear form and a pair of feature maps.

We consider multi-task learning in this paper. Many real-world applications involve learning multiple tasks. A standard methodology in machine learning is to learn one task at a time. Large problems are hence broken into small and reasonably independent subproblems that are learned separately and then recombined. Multi-task learning where the unknown target function to be learned from finite sample data is vector-valued appears more often in practice \cite{Micchelli05}. Learning multiple related tasks simultaneously can be more beneficial. For instance, in certain circumstances, data for each task are not enough to avoid over-fitting and hence results in poor generalization ability. In this case, what is learned for each task can help other related tasks be learned better. There are numerical experiments in the literature \cite{ARL2012,Caruana1997} which demonstrate that multi-task learning can lead to better generalization performance than learning each task independently. Recent progress about  multi-task learning in vector-valued RKHSs can be found in \cite{Caponnetto,Carmeli}. In such a framework, both the space of the candidate functions used for approximation and the output space are chosen as Hilbert spaces. Mathematical theory of learning on vector-valued RKBSs based on semi-inner-products has been proposed in \cite{ZZ13}. The spaces considered there are reflexive and thus do not accommodate the $\ell^1$ norm.

Motivated by sparse multi-task learning, we shall construct vector-valued RKBSs with the $\ell^1$ norm in this paper.
To ensure that the existence of a reproducing kernel, the construction starts directly with an admissible multi-task kernel satisfying three assumptions: non-singularity, boundedness, and independence. Then, we are able to obtain a vector-valued RKBS with the $\ell^1$ norm and its associated reproducing kernel.

Moreover, we will investigate the regularization model in such spaces. The classical linear representer theorem is a key to the mathematical analysis of kernel methods in machine learning \cite{Scholkopf2001}. It asserts that the minimizer is a linear combination of the kernel functions at the sampling points.  The representer theorem in scalar-valued RKHSs was initially established by Kimeldorf and Wahba \cite{Kimeldorf1971}. The result was generalized to other regularizers in \cite{SHS01}.
Recent references \cite{FHY15,LinUnifRKBS,XY18,Ye13,ZXZ09,ZZ12} developed representer theorems for various scalar-valued RKBSs. We shall present the representer theorem for machine learning schemes in vector-valued RKBSs with the $\ell^1$ norm. We shall see that this is equivalent to requiring the Lebesgue constant  of the admissible multi-task kernel to be exactly bounded by $1$. To accommodate more kernels, we consider a relaxed representer theorem.

The outline of the paper is as follow.
In Section \ref{sec:construction}, we present definitions of vector-valued RKBSs, the associated reproducing kernels, and admissible multi-task kernels.
We next start constructing RKBSs of vector-valued functions with the $\ell^1$ norm.
Section \ref{sec:representer} establishes representer theorems for minimal norm interpolation and regularization networks in the constructed spaces.
Examples of admissible multi-task kernels are given in Section \ref{sec:examples}. To accommodate more kernel functions, a relaxed version of linear representer theorem is discussed in Section \ref{sec:relaxed}. In the last section, numerical experiments for both synthetic data and real-world benchmark data are presented to demonstrate the advantages of the proposed construction and regularization models.

\section{Construction of vector-valued RKBSs with the \mathinhead{\ell^1}{l1} norm }\label{sec:construction}
We shall present the construction of vector-valued RKBSs with the $\ell^1$ norm in this section. Specifically, we will first introduce the definition of general vector-valued RKBSs and then construct the specific vector-valued RKBSs with the $\ell^1$ norm.

To give a formal definition of vector-valued RKBSs in our setting, we first review the definition of Banach spaces of vector-valued functions. A normed vector space $V$ of functions from $X$ to $Y\subseteq\bR^d$ is called a {\it Banach space of vector-valued functions} if it is a Banach space whose elements are vector-valued functions on $X$ and for each $f\in V$, $\|f\|_{V}=0$ if and only if $f(x)={\bf 0}$ for all $x\in X$. Here, $\vzero$ denotes the zero vector of $\bR^d$.
For instance, $L^p([0,1])$, $1\le p<+\infty$ is not a Banach space of functions while $C([0,1])$ is. The definition of general vector-valued RKBSs is presented below.

\begin{definition}[Vector-valued RKBS]\label{def:vvRKBS}
Let $X$ be a prescribed nonempty set, and let $Y$ be a Banach space.
A vector-valued RKBS $\cB$ of functions from  $X$ to $Y$ is a Banach space of certain vector-valued functions $f: X\to Y$ such that every point evaluation functional $\delta_x$, $x\in X$ on $\cB$ is continuous. That is, for any $x\in X$, there exists a constant $C_x>0$ such that
$$
\|\delta_x(f)\|_{Y}=\|f(x)\|_Y\le C_x\|f\|_{\cB}\mbox{ for all }f\in\cB.
$$
\end{definition}
Definition \ref{def:vvRKBS} is a natural ``vectorized'' generalization of the scalar-valued RKHS \cite{Aronszajn,Scholkopf2001} and the scalar-valued RKBS in \cite{LinUnifRKBS}. In \cite{Micchelli05}, a Hilbert space $\cH$ from $X$ to a Hilbert space $Y$ with inner product $\langle \cdot,\cdot\rangle_Y$ is called an RKHS of  vector-valued functions if for any $y\in Y$ and $x\in X$, the linear functional  which maps $f\in \cH$ to $\langle y,f(x)\rangle_Y$ is continuous. With the tool of semi-inner product, reference \cite{ZZ13} initially proposed the notion of vector-valued RKBS for multi-task learning in 2013. The prerequisite is that $\cB$ and $Y$ are uniform Banach spaces. Those requirements more or less seem unnatural. We are able to remove them by exploiting the definition of reproducing kernels via continuous bilinear forms.

We remark that there are no kernels directly mentioned in the above definition of the general vector-valued RKBSs. We will introduce a definition of the associated reproducing kernel through bilinear forms. Recall that a bilinear form between two normed vector spaces $V_1$ and $V_2$ is a function $(\cdot,\cdot)_{V_1\times V_2}$ from $V_1\times V_2$ to $\bR$ that is linear about both arguments. It is said to be a {\it continuous bilinear form} if there exists a positive constant $C$ such that
$$
|\langle f,g\rangle_{V_1\times V_2}|\le C\|f\|_{V_1}\|g\|_{V_2}\mbox{ for all }f\in V_1, g\in V_2.
$$

From now on, we assume the output space $Y=\bR^d$. A vector $\vc\in Y$ is always viewed as a $d\times 1$ column vector and we denote by $\vc^{\top}$ its transpose.
\begin{definition}[Reproducing Kernel]\label{def:RepKernel}
Let $X$ be a nonempty set, and let $\cB$ be a vector-valued RKBS from $X$ to $\bR^d$. If there exists a Banach space $\cB^{\#}$ of vector-valued functions  from $X$ to $\bR^d$, a continuous bilinear form $\langle \cdot,\cdot\rangle_{\cB\times\cB^\#}$, and a matrix-valued function $\bK: X\times X\to\bR^{d\times d}$ such that $\bK(\cdot,x)\vc\in\cB^\#$ for all $x\in X$ and $\vc\in\bR^d$, and
\begin{equation}\label{RPeq1}
(f,\bK(\cdot,x)\vc)_\bK=f(x)^{\top}\vc\ \mbox{ for all }  x\in X, \vc\in\bR^d, f\in\cB,
\end{equation}
then we call $\bK$ a reproducing kernel for $\cB$. If in addition, $\cB^\#$ is also a vector-valued RKBS, $\bK(x,\cdot)\vc\in\cB$ for all $x\in X$ and $\vc\in\bR^d$, and
\begin{equation}\label{RPeq2}
(\bK(x,\cdot)\vc,g)_\bK=\vc^{\top}g(x) \mbox{ for all } x\in X, \vc\in\bR^d, g\in\cB^{\#},
\end{equation}
then we call $\cB^\#$ an adjoint vector-valued RKBS of $\cB$, and call $\cB$ and $\cB^\#$ a  pair of vector-valued RKBSs. In the latter case, $\tilde{K}(x,x'):=K(x',x)$ for $x,x'\in X$, is a reproducing kernel for $\cB^\#$.
\end{definition}
We call (\ref{RPeq1}) and (\ref{RPeq2}) the {\it reproducing properties} for the kernel $\bK$ in vector-valued RKBSs $\cB$ and $\cB^\#$.

We shall next construct the specific vector-valued RKBSs with the $\ell^1$ norm satisfying the above conditions of general vector-valued RKBSs. The construction is built on certain multi-task kernels. To this end, we first introduce admissible multi-task kernels and some related notations. For any vector $\vu$ and $p\in [1,\infty]$, we use $\|\vu\|_p$ to denote the $\ell^p$ norm of $\vu$. For any matrix $\mA$ and $p\in [1,\infty]$, we use $\|\mA\|_p$ to denote the $\ell^p$-induced matrix norm of $\mA$. We denote for any nonempty set $\Omega$ by $\ell^1_d(\Omega)$ the Banach space of vector-valued functions on $\Omega$ that is integrable with respect to the counting measure on $\Omega$. Specifically,
\begin{equation}\label{ell1Omega}
\ell^1_d(\Omega):=\Big\{\vc=(\vc_t\in\bR^d:t\in\Omega):\|{\bf c}\|_{\ell^1_d(\Omega)}=\sum_{t\in\Omega}\|\vc_t\|_1<+\infty\Big\}.
\end{equation}
Note that $\Omega$ might be uncountable, but for every ${\bf c}\in\ell_d^1(\Omega)$, the support $\supp \vc:=\{t\in\Omega: \vc_t\neq \vzero\}$ must be at most countable. Let us denote $\bN_m:=\{1,2,\dots,m\}$ for any $m\in\bN$.

\begin{definition}[Admissible Multi-task Kernel]\label{def:AdmissibleKernel} Let $X$ be a nonempty set and let
$\bK:X\times X\to\bR^{d\times d}$ be a matrix-valued function such that $\bK^{\top}=\bK$. Such a kernel is an admissible multi-task kernel if the following assumptions are satisfied:
\begin{description}
  \item[(A1)] (Non-singularity) for all $m\in\bN$ and all pairwise distinct sampling points $\bx=\{x_j:j\in\bN_m\}\subseteq X$, the matrix
  $$
  \bK[\bx]:=\Big[\bK(x_k,x_j): j,k\in\bN_m\Big]\in \bR^{md\times md}
  $$
  is non-singular;
  \item[(A2)] (Boundedness) there exists $\kappa>0$ such that $\|\bK(x,x')\|_1\le \kappa$ for all $x,x'\in X$;
  \item[(A3)] (Independence) for all pairwise distinct points $x_j\in X$, $j\in\bN$ and $(\vc_j\in\bR^d:j\in\bN)\in \ell_d^1(\bN)$, if $\sum_{j\in\bN}\bK(x_j,x)\vc_j=\vzero$ for all $x\in X$ then $\vc_j=\vzero$ for all $j\in\bN$.
\end{description}
\end{definition}

We now present the construction of vector-valued RKBSs with the $\ell^1$ norm based on admissible multi-task kernels. Suppose that $\bK:X\times X\to\bR^{d\times d}$ is an admissible multi-task kernel defined above. We then define
\begin{equation}\label{eq:defBK}
\cB_\bK:=\Big\{\sum_{x\in\supp \vc}\bK(x,\cdot)\vc_x: \vc=(\vc_x\in\bR^d:x\in X)\in\ell_d^1(X)\Big\}
\end{equation}
with the norm
\begin{equation}\label{NormbK}
\Big\|\sum_{x\in\supp{\vc}}\bK(x,\cdot)\vc_x\Big\|_{\cB_\bK}:=\|\vc\|_{\ell^1_d(X)},
\end{equation}
where $\ell_d^1(X)$ is given as in (\ref{ell1Omega}).  By ({\bf A3}), we should point out that the norm given by (\ref{NormbK}) is well-defined. In other words, for any $f\in\cB_\bK$, $\|f\|_{\cB_\bK}=0$ if and only if $f=\vzero$ everywhere on $X$.

We next show that the Banach space of functions $\cB_\bK$ defined above is a vector-valued RKBS on $X$ according to Definition \ref{def:vvRKBS}.

\begin{proposition}
If $\bK$ is an admissible multi-task kernel, then the space $\cB_\bK$ as defined in Equation (\ref{eq:defBK}) is a vector-valued RKBS on $X$ in the sense that
$$
\|f(x)\|_1\le \kappa\|f\|_{\cB_\bK}, \mbox{ for all } x\in X, f\in\cB_\bK.
$$
\end{proposition}
\begin{proof}
Note that $\ell^1_d(X)$ is a Banach space. By definition (\ref{NormbK}) of the norm on $\cB_\bK$, $\cB_\bK$ is a vector-valued Banach space on $X$. For any $f\in\cB_\bK$, there exists $\vc\in \ell_d^1(X)$ such that
$$
f=\sum_{t\in\supp{\bf c}}\bK(t,\cdot)\vc_t.
$$
For any $x\in X$, by Assumption ({\bf A2}) and Equation (\ref{NormbK}), we compute
$$
\|\delta_x(f)\|_1=\|f(x)\|_1\le \sum_{t\in\supp{\vc}}\Big\|\bK(t,x)\vc_t\Big\|_1\le \sum_{t\in\supp{\vc}}\Big\|\bK(t,x)\Big\|_1\|\vc_t\|_1\le \kappa\sum_{t\in\supp{\vc}}\|\vc_t\|_1= \kappa\|f\|_{\cB_\bK}
$$
for all $f\in\cB_\bK$. In other words, the point evaluation functional $\delta_x$, $x\in X$ is continuous on $\cB_\bK$ in the sense that
$\|\delta_x(f)\|_1\le \kappa\|f\|_{\cB_\bK}$ for all $f\in\cB_\bK$. The proof is hence complete.
\end{proof}

We next show that $\bK$ is a reproducing kernel of $\cB_\bK$ through checking the conditions in Definition \ref{def:RepKernel}. For this purpose, we introduce an adjoint vector-valued RKBS $\cB_\bK^\#$ below. Let
\begin{equation*}
\cB_0^{\#}:=\Big\{\sum_{k=1}^{n} \bK(\cdot,x_k)\vb_k: x_k\in X, \vb_k\in\bR^d, k\in\bN_n \mbox{ for all }n\in\bN\Big\}
\end{equation*}
endowed with the supremum norm
\begin{equation*}
\Big\|\sum_{k=1}^{n} \bK(\cdot,x_k)\vb_k\Big\|_{\cB^\#_0}:=\sup_{t\in X}\Big\|\sum_{k=1}^{n} \bK(t,x_k)\vb_k\Big\|_\infty.
\end{equation*}
We point out that an abstract completion of $\cB_0^\#$ might not consist of functions. To this end, we let $\cB_\bK^{\#}$ be the completion of $\cB_0^{\#}$ under the supremum norm by the Banach completion process described below.  Suppose $\{g_n:n\in\bN\}$ is a Cauchy sequence in $\cB_0^\#$.
Observe that point evaluation functionals $\delta_x$, $x\in X$ are continuous on $\cB_0^\#$ in the sense that
$$
\|g(x)\|_{\infty}\le \|g\|_{\cB_0^\#}\mbox{ for all }g\in\cB_0^\#.
$$
Consequently, for any $x\in X$, the sequence $\{g_n(x):n\in\bN\}$ converges in $\bR^d$.  We define the limit by $g(x)$, which is a vector-valued function on $X$.
Clearly, two equivalent Cauchy sequences in $\cB_0^\#$ give the same function. We let $\cB_\bK^\#$ be consisting of all such limit functions $g$ with the norm
$$
\|g\|_{\cB_\bK^\#}:=\lim_{n\to+\infty}\|g_n\|_{\cB_0^\#}=\lim_{n\to+\infty}\sup_{x\in X}\|g_n(x)\|_{\infty}.
$$
It follows that for any $g\in\cB_\bK^\#$,
$$
\|g\|_{\cB_\bK^\#}:=\sup_{x\in X}\|g(x)\|_\infty.
$$

Based on the above construction, we immediately have that $\cB_\bK^\#$ defined above is also a vector-valued RKBS.
\begin{proposition}
If $\bK$ is an admissible multi-task kernel, then the space $\cB^\#_\bK$ is a vector-valued Banach spaces in the sense that
\begin{equation*}
\|g(x)\|_\infty\le \|g\|_{\cB_\bK^\#} \mbox{ for all } x\in X, g\in \cB_\bK^\#.
\end{equation*}
\end{proposition}

We next characterize the reproducing properties in $\cB_\bK$ and $\cB_\bK^\#$ via a bilinear form. To this end, we define the linear space $\cB_0$ by
\begin{equation}\label{B0}
\cB_0:=\Big\{\sum_{j=1}^{m} \bK(x_j,\cdot)\vc_j:x_j\in X, \vc_j\in\bR^d,m\in\bN\Big\}.
\end{equation}
Clearly, $\cB_\bK$ is a Banach completion of $\cB_0$ under the $\ell^1$ norm.
We define a bilinear form $(\cdot,\cdot)_\bK$ on $\cB_0\times\cB_0^\#$  by
\begin{equation*}
\Big(\sum_{j=1}^m \bK(x_j,\cdot)\vc_j,\sum_{k=1}^{n}\bK(\cdot,s_k)\vb_k\Big)_{\bK}=\sum_{k=1}^{n}\sum_{j=1}^m  \vc_j^{\top}\bK(x_j,s_k)\vb_k,\  s_j,t_k\in X,\, \vc_j, \vb_k\in\bR^d.
\end{equation*}
According to the norms on $\cB_\bK$ and $\cB_0^\#$, we obtain
\begin{eqnarray*}
\Big|\Big(\sum_{j=1}^m \bK(x_j,\cdot)\vc_j,\sum_{k=1}^{n}\bK(\cdot,s_k)\vb_k\Big)_{\bK}\Big|
&\le& \sum_{j=1}^m \Big|\vc_j^{\top} \Big(\sum_{k=1}^{n} \bK(x_j,s_k)\vb_k\Big)\Big|\\
&\le& \sum_{j=1}^m \|\vc_j\|_1 \Big\|\sum_{k=1}^{n} \bK(x_j,s_k)\vb_k\Big\|_\infty\\
&\le& \Big(\sum_{j=1}^m \|\vc_j\|_1\Big) \Big(\sup_{t\in X}\Big\|\sum_{k=1}^{n} \bK(t,s_k)\vb_k\Big\|_\infty\Big)\\
&=&\Big\|\sum_{j=1}^m \bK(x_j,\cdot)\vc_j\Big\|_{\cB_0}\Big\|\sum_{k=1}^{n}\bK(\cdot,s_k)\vb_k\Big\|_{\cB_0^\#}.
\end{eqnarray*}
It implies that the bilinear form $(\cdot,\cdot)_\bK$ is continuous on $\cB_0\times\cB_0^\#$.
By applying the Hahn-Banach extension theorem twice, the bilinear form can be extended to $\cB_\bK\times\cB_\bK^\#$ such that
$$
|\langle f,g\rangle_{\bK}|\le \|f\|_{\cB_\bK}\|g\|_{\cB_\bK^\#}\mbox{ for all }f\in\cB_\bK, g\in\cB_\bK^\#.
$$

Finally, we are ready to show that $\bK$ is a reproducing kernel for $\cB_\bK$.

\begin{theorem}  If $\bK:X\times X\to\bR^{d\times d}$ is an admissible multi-task kernel then $\cB_\bK$ and $\cB_\bK^\#$ are a pair of vector-valued RKBSs, and $\bK$ is a reproducing kernel for $\cB_\bK$ and $\cB_\bK^\#$.
\end{theorem}
\begin{proof}
It is sufficient to verify the reproducing properties for $\bK$ in $\cB_\bK$ and $\cB_\bK^\#$.
Recall that $(\cdot,\cdot)_\bK$ is a continuous bilinear form on $\cB_\bK\times\cB_\bK^{\#}$. For any $f\in\cB_\bK$, there exist distinct points $x_j\in X$, $j\in\bN$ and ${\bf c}\in\ell^1_d(\bN)$ such that
$f=\sum_{j\in\bN}\bK(x_j,\cdot)\vc_j$. Since $\bK=\bK^{\top}$, it follows from a direct computation
$$
\begin{array}{ll}
\displaystyle{ (f,\bK(\cdot,x)\vb)_\bK }
&\displaystyle{=\lim_{n\to+\infty}\Big(\sum_{j=1}^n \bK(x_j,\cdot)\vc_j,\bK(\cdot,x)\vb\Big)=\lim_{n\to+\infty}\sum_{j=1}^n \vc_j^{\top}\bK(x_j,x)\vb } \\
&\displaystyle{=\lim_{n\to+\infty}\Big(\sum_{j=1}^n \bK(x_j,x)\vc_j\Big)^{\top}\vb=f(x)^{\top}\vb }\\
\end{array}
$$
for any $x\in X$. The reproducing property for $\bK$ in $\cB_\bK^{\#}$ follows in a similar way.
\end{proof}

\section{Representer Theorems}\label{sec:representer}
The linear representer theorems play a fundamental role in regularized learning schemes in machine learning. It helps us to turn the infinite-dimensional optimization problem to an equivalent optimization problem in a finite-dimensional subspace. We shall establish in this section representer theorems for the minimal norm interpolation problem and regularization networks in the vector-valued RKBSs $\cB_\bK$ with the $\ell^1$ norm constructed in Section \ref{sec:construction}.

\subsection{Minimal Norm Interpolation}

A minimal norm interpolation problem in a vector-valued RKBS $\cB_\bK$ with respect to a set of sampling data $\{(x_j,\vy_j):j\in\bN_m\}\subseteq X\times\bR^d$ is to solve
\begin{equation}\label{eq:minnorminterp}
\min_{f\in \cI_{\bx}(\by)} \|f\|_{\cB_\bK}, \mbox{ where }\cI_{\bx}(\vy):=\Big\{f\in\cB_\bK: f(\bx)=\vy\Big\}.
\end{equation}
We always assume the minimizer of the above problem exists in this paper.

We say the vector-valued RKBS $\cB_\bK$ has the {\it linear representer theorem for minimal norm interpolation} if for any integer $m\in\bN$ and any choice of sampling data $\{(x_j,\vy_j):j\in\bN_m\}\subseteq X\times\bR^d$,  there always exists a minimizer of the problem \eqref{eq:minnorminterp} in the following finite-dimensional subspace:
\begin{equation}\label{eq:finitesubpace}
\cS^{\bx}:=\Big\{\sum_{j=1}^m\bK(x_j,\cdot)\vc_j: \vc_j\in\bR^d,\ j\in\bN_m\Big\}.
\end{equation}
We point out that the finite-dimensional subspace $\cS^{\bx}$ has dimension $md$.

We shall next investigate the linear representer theorem for minimal norm interpolation. We remark that the interpolation space $\cI_{\bx}(\vy)$ is infinite-dimensional in general. We need to show the minimal norm interpolation in $\cB_\bK$ is equivalent to that in the finite-dimensional subspace in $\cS^{\bx}$. To this end, we first show the minimal norm interpolation in a finite-dimensional subspace containing $\cS^{\bx}$ could be reduced to that in $\cS^{\bx}$. In particular, we consider the finite-dimensional subspace $\cS^{\tilde{\bx}}$, where $\tilde{\bx}:=\bx\cup \{x_{m+1}\}$ and $x_{m+1}\in X\setminus \bx$.

For notational simplicity, we denote
$$
\bK_\bx (t):=(\bK(t,x_j):j\in\bN_m)^{\top}\in\bR^{md\times d},\  \bK^\bx (t):=(\bK(x_j,t):j\in\bN_m)\in\bR^{d\times md},\ t\in X.
$$
It is worthwhile to point out that $\bK_\bx$ is in general not the transpose of $\bK^\bx$. If, in addition, $\bK$ is symmetric in the sense that $\bK(x,x')=\bK(x',x)$ for all $x,x'\in X$ then by $\bK=\bK^{\top}$ we have $\bK_\bx^{\top}=\bK^{\bx}$.

We will present a necessary and sufficient condition for the equivalence of the minimal norm interpolation in $\cS^{\tilde{\bx}}$ to that in $\cS^{\bx}$.

\begin{lemma}\label{lemmaA4}
Suppose the multi-task kernel $\bK$ is admissible. Let $\bx=\{x_j:j\in\bN_m\}$ be a set of pairwise distinct points, let $x_{m+1}$ be an arbitrary point in $X\setminus \bx$, and let $\tilde{\bx}=\bx\cup \{x_{m+1}\}$. Then
\begin{equation}\label{representoreq1}
\min_{f\in\cI_{\bx}(\vy)\cap \cS^{\bar{\bx}}}\|f\|_{\cB_\bK}=\min_{f\in\cI_{\bx}(\vy)\cap \cS^{\bx}}\|f\|_{\cB_\bK}, \quad \mbox{ for all }\vy\in \bR^{md}
\end{equation}
if and only if $\|\bK[\bx]^{-1}\bK_\bx(x_{m+1})\|_1\le 1$.
\end{lemma}
\begin{proof}
Notice that the set $\cI_{\bx}(\vy)\cap \cS^{\bx}$ possesses only one interpolant $f(t)=\bK^{\bx}(t)\bK[\bx]^{-1}\vy$, $t\in X$.
Let $\tilde{f}\in \cI_\bx(\vy)\cap \cS^{\tilde{\bx}}$ and $\vb:=\tilde{f}(x_{m+1})$.  Note that $\tilde{f}$ is uniquely determined by $\vb$ as it has already satisfied the interpolation condition $\tilde{f}(\bx)=\vy$. Specifically, by the admissible assumption (A1), we get $\tilde{f}(t)=\bK^{\tilde{\bx}}(t)\bK[\tilde{\bx}]^{-1}\tilde{\vy}$, $t\in X$,
where $\tilde{\vy}:=(\vy^{\top},\vb)^{\top}\in\bR^{(m+1)d}$.
It follows from a result (see for instance, \cite{Rasmussen2006}, pages 201-202) concerning the inversion of $2\times 2$ blockwise invertible matrix that
$$
\begin{array}{ll}
\displaystyle{\bK[\tilde{\bx}]^{-1}\tilde{\vy}}
&\displaystyle{=
\left[
  \begin{array}{cc}
    \bK[\bx] & \bK_\bx(x_{m+1}) \\
    \bK^\bx(x_{m+1}) & \bK(x_{m+1},x_{m+1}) \\
  \end{array}
\right]^{-1}
\left[
  \begin{array}{c}
    \vy \\
    \vb \\
  \end{array}
\right]}\\
&\displaystyle{=\left[
  \begin{array}{cc}
   \bK[\bx]^{-1}+\bK[\bx]^{-1}\bK_\bx(x_{m+1})\bp^{-1}\bK^\bx(x_{m+1})\bK[\bx]^{-1} & -\bK[\bx]^{-1}\bK_\bx(x_{m+1})\bp^{-1}\\
    -\bp^{-1}\bK^\bx(x_{m+1})\bK[\bx]^{-1} & \bp^{-1} \\
  \end{array}
\right]
\left[
  \begin{array}{c}
    \vy \\
    \vb \\
  \end{array}
\right]}\\
&\displaystyle{=
\left[
  \begin{array}{c}
    \bK[\bx]^{-1}\vy+  \bK[\bx]^{-1}\bK_\bx(x_{m+1})\bp^{-1}\bq \\
    -\bp^{-1}\bq \\
  \end{array}
\right]}\\
\end{array}
$$
where the matrix
$$
\bp:=\Big[\bK(x_{m+1},x_{m+1})-\bK^\bx(x_{n+1})\bK[\bx]^{-1}\bK_\bx(x_{m+1})\Big]\in \bR^{d\times d}
$$
is non-singular and the column vector $\bq$ is
$$
\bq:=\bK^\bx(x_{m+1})\bK[\bx]^{-1}\vy-\vb\in \bR^{d}.
$$

We now show sufficiency. If $\|\bK[\bx]^{-1}\bK_\bx(x_{m+1})\|_1\le 1$, then
$$
\begin{array}{ll}
\displaystyle{\|\tilde{f}\|_{\cB_\bK}}
\displaystyle{=\|\bK[\tilde{\bx}]^{-1}\tilde{\vy}\|_1}
&\displaystyle{\ge \|\bK[\bx]^{-1}\by+\bK[\bx]^{-1}\bK_\bx(x_{n+1})\bp^{-1}\bq\|_1+\|\bp^{-1}\bq\|_1}\\
&\displaystyle{\ge \|\bK[\bx]^{-1}\vy\|_1-\| \bK[\bx]^{-1}\bK_\bx(x_{n+1})\bp^{-1}\bq\|_1+\|\bp^{-1}\bq\|_1}\\
&\displaystyle{\ge \|\bK[\bx]^{-1}\vy\|_1-\| \bK[\bx]^{-1}\bK_\bx(x_{n+1})\|_1\|\bp^{-1}\bq\|_1+\|\bp^{-1}\bq\|_1  }\\
&\displaystyle{\ge \|\bK[\bx]^{-1}\vy\|_1}\\
&\displaystyle{=\|f\|_{\cB_\bK}}\\
\end{array}
$$
which implies
$$
\min_{f\in\cI_{\bx}(\vy)\cap \cS^{\tilde{\bx}}}\|f\|_{\cB_\bK}\ge \min_{f\in\cI_{\bx}(\vy)\cap \cS^{\bx}}\|f\|_{\cB_\bK}, \mbox{ for all }\vy\subseteq \bR^{md}.
$$
Since $\cS^{\bx}\subseteq\cS^{\tilde{\bx}}$, the reverse direction of this inequality holds.
Thus, (\ref{representoreq1}) holds true.

Conversely, if (\ref{representoreq1}) is true for all $\vy\in\bR^{md}$ then we must have
$$
\|\bK[\tilde{\bx}]^{-1}\tilde{\vy}\|_1\ge \|\bK[\bx]^{-1}\vy\|_1, \mbox{ for all } \vy\in\bR^{md}\mbox{ and all }  \vb\in\bR^d.
$$
Fix $j\in\bN_d$. In particular, if we choose
$$
\vy:=\bK_\bx (x_{m+1})\ve_j, \mbox{ and } \vb:=\bK^\bx(x_{m+1}) \bK[\bx]^{-1}\vy+\bp \ve_j,
$$
where $\ve_j$ is a column vector in $\bR^d$ whose $j$-th component is $1$ and other components are $0$, then
$$
\bq=-\bp \ve_j,\ \bp^{-1}\bq=-\ve_j\mbox{ and }\bK[\bx]^{-1}\by+  \bK[\bx]^{-1}\bK_\bx(x_{m+1})\bp^{-1}\bq={\vzero}_{md},
$$
where ${\vzero}_{md}$ denotes the zero column vector in $\bR^{md}$. Consequently
$$
\|\bK[\tilde{\bx}]^{-1}\tilde{\vy}\|_1=
\left\|
  \begin{array}{c}
    {\vzero}_{md} \\
    \ve_j \\
  \end{array}
\right\|_1=1
\mbox{ and }
\|\bK[\bx]^{-1}\vy\|_1=\|\bK[\bx]^{-1}\bK_{\bx}(x_{m+1})\ve_j\|_1.
$$
Recalling the definition of the $\ell^1$ norm of matrices, we get $\|\bK[\bx]^{-1}\bK_\bx(x_{m+1})\|_1\le 1$.
The proof is complete.
\end{proof}

We are now ready to present the following necessary and sufficient condition for the equivalence of the minimal norm interpolation in $\cB_\bK$ to that in $\cS^{\bx}$:

({\bf Lebesgue Constant Condition})  For all $m\in\bN$ and all pairwise distinct sampling points $\bx=\{x_j:j\in\bN_m\}\subseteq X$,
\begin{equation}\label{eq:A4}
\sup_{t\in X}\|\bK[\bx]^{-1}\bK_{\bx}(t)\|_1\le 1,
\end{equation}
where $\bK_\bx (t)=(\bK(t,x_j):j\in\bN_m)^{\top}$, $t\in X$.

\begin{theorem}\label{thm:repinterp}
Suppose $\bK$ is an admissible multi-task kernel. The space $\cB_\bK$ satisfies the linear representer theorem for minimal norm interpolation if and only if the Lebesgue constant condition \eqref{eq:A4} holds.
\end{theorem}
\begin{proof} We first prove the necessity. The space $\cB_\bK$ satisfies the linear representer theorem for minimal norm interpolation if and only if  for any integer $m$ and any choice of sampling data $\{(x_j,\vy_j):j\in\bN_m\}\subseteq X\times\bR^d$
$$
\min_{g\in\cI_{\bx}(\vy)}\|g\|_{\cB_\bK}=\min_{f\in\cI_{\bx}(\vy) \cap \cS^{\bx}}\|f\|_{\cB_\bK}.
$$
Choose a new point $x_{m+1}\in X\setminus\bx$. Observe $\cI_{\bx}(\vy) \cap \cS^{\bx}\subseteq \cI_{\bx}(\vy) \cap \cS^{\tilde{\bx}}\subseteq \cI_{\bx}(\vy)$.
By Lemma \ref{lemmaA4},  we have $\|\bK[\bx]^{-1}\bK_\bx(x_{m+1})\|_1\le 1$.  Observe that $\|\bK[\bx]^{-1}\bK_\bx(t)\|_1= 1$ for any $t\in\bx$.
Since we can pick an arbitrary integer $m$ and an arbitrary point $x_{m+1}$,  this leads to \eqref{eq:A4}.

We next show the sufficiency. Suppose the Lebesgue constant condition \eqref{eq:A4} holds. Fix $m\in\bN$.
Since $\cI_{\bx}(\by) \cap \cS^{\bx}\subseteq \cI_{\bx}(\by)$,
$$
\min_{g\in\cI_{\bx}(\by)}\|g\|_{\cB_\bK}\le \min_{f\in\cI_{\bx}(\by) \cap \cS^{\bx}}\|f\|_{\cB_\bK}.
$$
It remains to prove the reverse direction of this inequality.  For this purpose, we shall first show $\|g\|_{\cB_\bK}\ge \min_{f\in\cI_{\bx}(\vy) \cap \cS^{\bx}}\|f\|_{\cB_\bK}$ for all $g\in \cI_{\bx}(\by)\cap \cB_0$, where the space $\cB_0$ is defined by (\ref{B0}). By $\cI_{\bx}(\vy) \cap \cS^{\bx}\subseteq \cI_{\bx}(\vy)\cap \cB_0$,
the set $\cI_{\bx}(\vy)\cap \cB_0$ is non-empty. We write $g$ as $g=\sum_{j=1}^n \bK({x}_j,\cdot)\vc_j$ for some integer $n\geq m$ and distinct points $\{{x}_j: j\in\bN_n\}$. This can always be done by adding some sampling points, setting the corresponding coefficients to be zero, and relabeling if necessary. Let $\vy_j:=g({x}_j)$, $j\in\bN_n$. Then set
$$
\bu_l:=(\vy_j: j\in\bN_l)^{T},\mbox{ and } \bv_l:=\{{x}_j:j\in\bN_l\} \mbox{ for }l=m,m+1,\dots,n.
$$
Note that $\vy=\bu_m$ and $\bx=\bv_m$. By $g\in \cI_{\bv_n}(\bu_n) \cap \cS^{\bv_n}$, it follows that
$\|g \|_{\cB_\bK}\geq \min_{f\in \cI_{\bv_n}(\bu_n) \cap \cS^{\bv_n}} \| f\|_{\cB_\bK}$.
Since $\cI_{\bv_n}(\bu_n) \subseteq \cI_{\bv_{n-1}}(\bu_{n-1})$, we apply Lemma \ref{lemmaA4} to get
$$
\min_{f\in \cI_{\bv_n}(\bu_n) \cap \cS^{\bv_n}} \| f\|_{\cB_\bK}\geq  \min_{f\in \cI_{\bv_{n-1}}(\bu_{n-1}) \cap \cS^{\bv_n}} \| f\|_{\cB_\bK} = \min_{f\in \cI_{\bv_{n-1}}(\bu_{n-1}) \cap \cS^{\bv_{n-1}}} \| f\|_{\cB_\bK}.
$$
It follows that $\|g \|_{\cB_\bK}\geq \min_{f\in \cI_{\bv_{n-1}}(\bu_{n-1}) \cap \cS^{\bv_{n-1}}} \| f\|_{\cB_\bK}$. Repeating this process, we get
\begin{equation}\label{repinterp1}
\|g \|_{\cB_\bK}\geq \min_{f\in \cI_{\bv_m}(\bu_m) \cap \cS^{\bv_m}} \| f\|_{\cB_\bK}=\min_{f\in\cI_{\bx}(\vy) \cap \cS^{\bx}}\|f\|_{\cB_\bK}\mbox{  for all }g\in\cI_{\bx}(\vy)\cap \cB_0.
\end{equation}

Now let $g\in\cI_{\bx}(\vy)$ be arbitrary but fixed. Recall that the completion of $\cB_0$ with respect to the $\ell^1$ norm is $\cB_\bK$.
Then there exists a sequence of vector-valued functions $g_j\in\cB_0$, $j\in\bN$ that converges to $g$ in $\cB_\bK$. We let $f$ and $f_j$ be the function in $\cS^{\bx}$ such that $f(\bx)=\by$ and $f_j(\bx)=g_j(\bx)$, $j\in\bN$. They are explicitly given by
$$
f= \bK^{\bx}(\cdot)\bK[\bx]^{-1}g(\bx) \quad \mbox{and} \quad f_j= \bK^{\bx}(\cdot)\bK[\bx]^{-1}g_j(\bx), \ \ j\in\bN.
$$
Since $g_j$ converges to $g$ in $\cB_\bK$ and point evaluation functionals are continuous on $\cB_\bK$, $g_j(\bx)\to g(\bx)$ as $j\to+\infty$. As a result,
$$
\lim_{j\to +\infty} \|f -f_j\|_{\cB_\bK}=\lim_{j\to +\infty}\|\bK[\bx]^{-1}(g(\bx)-g_j(\bx))\|_1\le \|\bK[\bx]^{-1}\|_1\lim_{j\to+\infty}\|g(\bx)-g_j(\bx)\|_1=0.
$$
By \eqref{repinterp1}, $\| g_j\|_{\cB_\bK} \geq \| f_j\|_{\cB_\bK}$ for all $j\in\bN$. It follows that $\|g\|_{\cB_\bK}\ge \|f\|_{\cB_\bK}$ and thus,
$$
\min_{g\in\cI_{\bx}(\by)}\|g\|_{\cB_\bK}\ge \min_{f\in\cI_{\bx}(\by) \cap \cS^{\bx}}\|f\|_{\cB_\bK}.
$$
The proof is complete.
\end{proof}

\subsection{Regularization Networks}
We will present a representer theorem for regularization network in the following form:
\begin{equation}\label{MinimizationBK1}
\min_{f\in\cB_{\bK}} L(f(\bx),\by)+\lambda\phi(\|f\|_{\cB_{\bK}}),
\end{equation}
where $L:\bR^{md}\times\bR^{md}\to\bR_+$ is a continuous loss function with the property $L(\vt,\vt)=0$ for any $\vt\in\bR^{md}$, and $\phi:\bR_+\to\bR_+$ a non-decreasing continuous regularizer with $\lim_{t\to+\infty}\phi(t)=+\infty$. We will always assume a minimizer of the above model \eqref{MinimizationBK1} exists.

Similarly, we say the vector-valued RKBS $\cB_\bK$ has {\it the linear representer theorem for regularization network}
if for any integer $m$, any choice of sampling data $\{(x_j,\vy_j):j\in\bN_m\}$ and any $\lambda>0$, there exists a minimizer of (\ref{MinimizationBK1}) in the finite-dimensional subspace $\cS^\bx$ as defined in \eqref{eq:finitesubpace}.

We shall establish a necessary and sufficient condition such that the linear representer theorem for regularization network holds in the vector-valued RKBS $\cB_\bK$. This is achieved by showing the equivalence between the representer theorems for the minimal norm interpolation problem and for the regularization network problem.

\begin{theorem}\label{thm:equiv-rep}
Suppose $\bK$ is an admissible multi-task kernel. The space $\cB_\bK$ satisfies the linear representer theorem for regularization network (\ref{MinimizationBK1})  if and only if it does so for minimal norm interpolation \eqref{eq:minnorminterp}.
\end{theorem}
\begin{proof}
Suppose we are given some sampling data $\{(x_j,\vy_j):j\in\bN_m\}\subseteq X\times\bR^d$ for some $m\in \bN$. We first assume that the space $\cB_\bK$ satisfies the linear representer theorem for minimal norm interpolation. We want to prove
\begin{equation}\label{eq:linearrepreg}
\min_{f\in\cB_\bK}L(f(\bx),\by)+\lambda\phi(\|f\|_{\cB_\bK})=\min_{f\in\cS^{\bx}}L(f(\bx),\by)+\lambda\phi(\|f\|_{\cB_\bK}).
\end{equation}
Note that the left hand side above is always bounded above by the right hand side as $\cS^{\bx}\subseteq \cB_\bK$. We only need to show the reverse inequality. For any $f\in\cB_\bK$, we consider the following minimal norm interpolation problem
\begin{equation*}
\min_{g\in\cI_{\bx}(f(\bx))}\|g\|_{\cB_\bK}.
\end{equation*}
Since $\cB_\bK$ satisfies the linear representer theorem for minimal norm interpolation, there exists a minimizer $f_0\in \cS^{\bx}$ of the above problem. Note that $f\in \cI_{\bx}(f(\bx))$. It follows that $f_0(\bx)=f(\bx)$ and $\|f_0\|_{\cB_\bK}\le \|f\|_{\cB_\bK}$. As a result, $L(f_0(\bx),\by)=L(f(\bx),\by)$ but $\phi(\|f_0\|_{\cB_\bK})\le\phi(\|f\|_{\cB_\bK})$ as $\phi$ is nondecreasing. That is, for any $f\in \cB_\bK$, there exists a $f_0\in \cS^{\bx}$ such that
\begin{equation*}
L(f_0(\bx),\by)+\lambda\phi(\|f_0\|_{\cB_\bK}) \leq L(f(\bx),\by)+\lambda\phi(\|f\|_{\cB_\bK}),
\end{equation*}
which implies the right hand side of (\ref{eq:linearrepreg}) is bounded above by the left hand side of (\ref{eq:linearrepreg}).

Moreover, we will prove the existence of the minimizer of right hand side of (\ref{eq:linearrepreg}). Since $\lim\limits_{t\to+\infty}\phi(t)=+\infty$, there exists a positive constant $\alpha$ such that
$$
\min_{f\in\cS^{\bx}}L(f(\bx),\by)+\lambda\phi(\|f\|_{\cB_\bK})=\min_{f\in\cS^{\bx},\|f\|_{\cB_\bK}\le \alpha}L(f(\bx),\by)+\lambda\phi(\|f\|_{\cB_\bK}).
$$
Note that the functional we are minimizing is continuous on $\cB_\bK$ by the assumption on $V$, $\phi$ and by the continuity of point evaluation functionals on $\cB_\bK$. By the elementary fact that a continuous function on a compact metric space attains its minimum in the space, the right hand side of (\ref{eq:linearrepreg}) has a minimizer that belongs to $\{f\in\cS^{\bx}:\|f\|_{\cB_\bK}\le \alpha\}$.

We next show the contrary part. That is, assuming $\cB_\bK$ satisfies the linear representer theorem for regularization networks, we need to show it also does so for minimal norm interpolation. We will find a minimizer of the minimal norm interpolation problem \eqref{eq:minnorminterp} explicitly. To this end, consider the regularization network \eqref{MinimizationBK1} with the following choices of $L$ and $\phi$:
\begin{equation*}
L(f(\bx),\by)=\|f(\bx)-\by\|^2_2, \quad\mbox{and}\quad \phi(t)=t.
\end{equation*}
For any $\lambda>0$, let $f_{0,\lambda}$ be a minimizer of (\ref{MinimizationBK1}) with the above choices of $L$ and $\phi$ in $\cS^{\bx}$. We could then write $f_{0,\lambda} = \bK^{\bx}(\cdot) \bc_\lambda$ for some $\bc_\lambda \in\bR^{md}$. It follows
\begin{equation*}
\|\bK[\bx] \bc_\lambda -\by\|^2_2=\|f_{0,\lambda}(\bx)-\by\|^2_2 \le L(f_{0,\lambda},\by) + \lambda \phi(\|f_{0,\lambda}\|_{\cB_\bK})\le L(0,\by)+\lambda\phi(\|0\|_{\cB_\bK})=\|\by\|_2^2.
\end{equation*}
Since $\bK[\bx]$ is nonsingular, the above inequality implies that $\{\bc_\lambda : \lambda>0\}$ forms a bounded set in $\bR^{md}$. By restricting to a subsequence if necessary, we may hence assume that $\bc_\lambda$ converges to some $\bc_0\in \bR^{md}$ as $\lambda$ goes to infinity. We then define $f_{0,0}:=\bK^{\bx}(\cdot)\bc_0$. It is clear that $f_{0,0}\in \cS^{\bx}$. We will show that $f_{0,0}$ is a minimizer of the minimal norm interpolation problem \eqref{eq:minnorminterp}.

Assume $g$ is an arbitrary interpolant in $\cI_{\bx}(\by)$. It is enough to show $f_{0,0}\in \cI_{\bx}(\by)$ and $\|f_{0,0}\|_{\cB_\bK} \leq  \|g \|_{\cB_\bK}$. By the definition of $f_{0,\lambda}$, we have
\begin{equation}\label{reginterpeq3}
\|f_{0,\lambda}(\bx)-\by\|^2_2 + \lambda\|f_{0,\lambda}\|_{\cB_\bK}\le \|g(\bx)-\by\|^2_2 + \lambda\|g\|_{\cB_\bK} =\lambda\|g\|_{\cB_\bK}.
\end{equation}
We observe that
\begin{equation}\label{eq:reginterp1}
\lim_{\lambda\to0}\|f_{0,\lambda}-f_{0,0}\|_{\cB_\bK}=\lim_{\lambda\to0}\| \bc_\lambda-\bc_0\|_1=0.
\end{equation}
Since point evaluation functionals are continuous on $\cB_\bK$, we have $f_{0,0}({x}_j)= \lim_{\lambda\to0} f_{0,\lambda}({x}_j) \mbox{ for all }j\in\bN_m$. Letting $\lambda\to 0$ on both sides of the above inequality \eqref{reginterpeq3}, we obtain $\|f_{0,0}(\bx)-\by\|^2_2=0$. That is, $f_{0,0}\in \cI_{\bx}(\by)$. Moreover, it also follows from \eqref{reginterpeq3} that $\|f_{0,\lambda}\|_{\cB_\bK} \le  \|g \|_{\cB_\bK}$ for all $\lambda>0$. This together with (\ref{eq:reginterp1}) implies $\|f_{0,0}\|_{\cB_\bK} \leq  \|g \|_{\cB_\bK}$, which finishes the proof.
\end{proof}

\begin{corollary} Suppose that $\bK$ is an admissible multi-task kernel. The space $\cB_\bK$ satisfies the linear representer theorem for regularization network (\ref{MinimizationBK1})  if and only if  $\bK$ satisfies the Lebesgue constant condition \eqref{eq:A4}.
\end{corollary}
\begin{proof}
This is an immediate consequence of Theorems \ref{thm:repinterp} and \ref{thm:equiv-rep}.
\end{proof}

\section{Examples of Admissible Multi-task Kernels}\label{sec:examples}
We shall present a few examples of admissible multi-task kernels in this section. In particular, we will investigate whether they satisfy the Lebesgue constant condition \eqref{eq:A4} such that the corresponding vector-valued RKBSs have the linear representer theorems. Positive together with some negative examples will both be given.

Consider a widely used class of multi-task kernels \cite{ARL2012,Caponnetto} in machine learning, which take the following form:
\begin{equation}\label{multitaskkernel}
\bK(x,x')=K(x,x')\bA,\ x,x'\in X,
\end{equation}
where $K:X\times X\to\bR$ is a scalar-valued positive definite kernel and $\bA$ denotes a $d\times d$ strictly positive definite symmetric matrix.
We reserve the notation $\bK$ in boldface type to denote the $d\times d$ matrix-valued function and $K$ the scalar-valued function as usual, respectively.

\begin{theorem} The kernel $\bK$ defined by (\ref{multitaskkernel}) is an admissible multi-task kernel if and only if $K$ is an admissible single-task kernel.
\end{theorem}
\begin{proof} Notice that $\bK^{\top}=\bK$ as $\bA$ is a  symmetric matrix.
We shall verify the non-singularity, boundedness, and independence assumptions. By  (\ref{multitaskkernel}), the non-singularity condition holds by noting that
$$
\bK[\bx]=[K(x_k,x_j)\bA: j,k\in\bN_m]=\diag(\bA,\dots,\bA)[K(x_k,x_j): j,k\in\bN_m].
$$
Observe that $\|\bK(x,x')\|_1=|K(x,x')|\|\bA\|_1$, $x,x'\in X$. As a result, $K$ is bounded on $X\times X$ if and only if $\|\bK(x,x')\|_1$ is.
To prove the independence assumption, we let $x_j\in X$, $j\in\bN$ be distinct points and $\vc=(\vc_j:j\in\bN)\in \ell_d^1(\bN)$. One sees that
$$
\sum_{j\in\bN}\bK(x_j,x)\vc_j=\sum_{j\in\bN}K(x_j,x)\bA \vc_j=\bA\sum_{j\in\bN}K(x_j,x)\vc_j.
$$
As $\bA$ is nonsingular,  $\sum_{j\in\bN}\bK(x_j,x)\vc_j={\bf 0}$ for all $x\in X$ if and only if
$\sum_{j\in\bN}K(x_j,x)(\vc_j)_k=0$ for all $k\in\bN_d$ and all $x\in X$, where $(\vc_j)_k$ denotes the $k$-th component of the column vector $\vc_j$. The proof is complete.
\end{proof}

The above theorem provides a way of constructing admissible multi-task kernels via their single-task counterparts. So far there are two admissible single-task kernels found in the literature. They are the Brownian bridge kernel $K(x,x'):=\min\{x,x'\}-xx'$, $x,x'\in(0,1)$ and the exponential kernel $K(x,x'):=e^{-|x-x'|}$, $x,x'\in\bR$. Here we are able to contribute another one. Specifically, we shall show that the {\it covariance of Brownian motion} (\cite{MortersPeres10}, Subsection 1.4) defined by
\begin{equation}\label{Brownian}
K(x,y):=\min\{x,y\},\ x,y\in (0,1),
\end{equation}
is an admissible single-task kernel. The corresponding RKHS $\cH_K$, also called the Cameron-Martin Hilbert space, consists of continuous functions $f$ on $[0,1]$ such that their distributional derivatives $f'\in L^2([0,1])$ and $f(0) = 0$. The inner product on $\cH_K$ is defined by
$\langle f,g\rangle_{\cH_K}:=\int_0^1 f'(x)g'(x)dx$, where $f,g\in\cH_K$.

To verify the Lebesgue constant condition \eqref{eq:A4} for this kernel, we explore the connection between the Lebesgue constants of kernels $K$ and $\bK$.
\begin{lemma}
Let $\bx=\{x_j\in X: j\in\bN_m\}$ be a set of distinct points, and $\alpha>0$. Then
$$
 \sup_{t\in X}\|K[\bx]^{-1}K_\bx(t)\|_1\le \alpha\mbox{  if and only if  } \sup_{t\in X}\|\bK[\bx]^{-1}\bK_\bx(t)\|_1\le \alpha.
$$
\end{lemma}
\begin{proof} For each $t\in X$, set $K[\bx]^{-1}K_\bx(t):=(b_1(t),b_2(t),\dots, b_d(t))^{\top}$.
By (\ref{multitaskkernel}), we compute
\begin{eqnarray}
  \bK[\bx]^{-1}\bK_\bx(\cdot) & =& \big[K(x_j,x_k)\bA:k,j\in\bN_m\big]^{-1}(K(x_j,\cdot)\bA:j\in\bN_m)^{\top} \nonumber\\
   &=& \big[K(x_j,x_k)\bI:k,j\in\bN_m\big]^{-1}\diag(\bA,\dots,\bA)^{-1}\diag(\bA,\dots,\bA)(K(x_j,\cdot)\bI:j\in\bN_m)^{\top} \nonumber\\
   &=& \big[K(x_j,x_k)\bI:k,j\in\bN_m\big]^{-1}(K(x_j,\cdot)\bI:j\in\bN_m)^{\top}  \label{Lebesgueeq3}
\end{eqnarray}
where $\diag(\bA,\dots,\bA)$ is a block diagonal matrix with $\bA$ as the diagonal entries. By (\ref{Lebesgueeq3}), this leads to
$$
 \sup_{t\in X}\|\bK[\bx]^{-1}\bK_\bx(t)\|_1=\max_{t\in X}\|(b_1(t)\bI,b_2(t)\bI,\dots, b_d(t)\bI)^{\top}\|_1.
$$
The proof is completed by noting the definition of the norm $\|\cdot\|_1$ for a matrices.
\end{proof}

\begin{theorem} The covariance of Brownian motion defined by (\ref{Brownian}) is an admissible single-task kernel and satisfies the Lebesgue constant condition \eqref{eq:A4}.
\end{theorem}
\begin{proof} Obviously, $|K(x,y)|$ is bounded by $1$ for all $x,y\in(0,1)$.  Let $m\in\bN$. Without loss of generality, we choose $0<x_1<x_2<\dots<x_m<1$ and let $\bx:=\{x_1,x_2,\dots,x_m\}$.
An easy computation shows that the determinant of the kernel matrix
$$
K[\bx]:=\big[\min\{x_j,x_k\}: j,k\in\bN_m\big]=\left[\begin{array}{cccccc}
x_1&x_1 &x_1&\dots&x_1 &x_1\\
x_1&x_2 &x_2&\dots&x_2 &x_2 \\
x_1&x_2 &x_3&\dots&x_3 &x_3\\
\vdots &\vdots &\vdots&\ddots&\vdots &\vdots \\
x_1&x_2 &x_3&\dots&x_{m-1} &x_{m-1}\\
x_1&x_2 &x_3&\dots&x_{m-1} &x_m\\
\end{array}
\right]
$$
is $x_1(x_2-x_1)(x_3-x_1)\cdots (x_m-x_1)\ne0$, and thus non-singularity assumption ({\bf A1}) holds.
Notice that for any $t\in(0,1)$,
$$
K(x,t):=\min\{x,t\}=\left\{\begin{array}{ll}
x& 0<x\le t,\\
t&t<x<1,
\end{array}
\right.
x\in (0,1).
$$
It follows that $K$ satisfies the independence assumption ({\bf A3}).

There are three cases when we compute the Lebesgue constant:

Case 1: If $0<t<x_1$ then $K_{\bx}(t)=(t,t,\dots,t)^{\top}$ and $K[\bx]^{-1}K_{\bx}(t)=(\frac{t}{x_1},0,\dots,0)^{\top}$.

Case 2: If $x_m<t<1$ then $K_{\bx}(t)=(x_1,x_2,\dots,x_m)^{\top}$ and $K[\bx]^{-1}K_{\bx}(t)=(0,0,\dots,0,1)^{\top}$.

Case 3: If $x_j\le t<x_{j+1}$ for some $j\in\bN_{m-1}$ then $K_{\bx}(t)=(x_1,x_2,\dots,x_j,t,\dots,t)^{\top}$ and
$$
K[\bx]^{-1}K_{\bx}(t)=\Big(0,0,\dots, 0,\frac{x_{j+1}-t}{x_{j+1}-x_j},\frac{t-x_j}{x_{j+1}-x_j},0,\dots,0\Big)^{\top}.
$$
In all three cases, it is straightforward to see that $\max_{t\in(0,1)}\|K[\bx]^{-1}K_{\bx}(t)\|_1\le 1$. Namely, the Lebesgue constant condition \eqref{eq:A4} is satisfied.
The proof is hence complete.
\end{proof}

At the end of this section, we give two negative examples. We shall show that neither the exponential kernel or the Gaussian kernel is admissible when the dimension is higher than $1$. This forces us to look for a relaxed linear representer theorem in the next section.

\begin{theorem} The multivariate exponential kernel $K(x,x')=e^{-\|x-x'\|_1}$, $x,x'\in\bR^n$ does not satisfy the Lebesgue constant condition \eqref{eq:A4} when $n\ge 2$.
\end{theorem}
\begin{proof} We begin with the proof of the special case when $n=2$. Choose three distinct points $x_1=(0,0)^{\top}$, $x_2=(1/2,0)^{\top}$, $x_3=(0,1/2)^{\top}$ in $\bR^2$. Let $\bx:=\{x_1,x_2,x_3\}$. Then we estimate the Lebesgue constant of bivariate exponential kernel
$$
\begin{array}{ll}
\displaystyle{\supp_{t\in\bR^2}\big\|K[\bx]^{-1}K_\bx(t)\big\|_1}
&\displaystyle{\ge \Big\|K[\bx]^{-1}K_\bx((1/2,1/2)^{\top})\Big\|_1 }\\
&\displaystyle{=\left\|\left[\begin{array}{ccc}1&e^{-\frac12}&e^{-\frac12}\\ e^{-\frac12}&1&e^{-1}\\e^{-\frac12}&e^{-1}&1 \end{array}\right]^{-1}
\left[\begin{array}{c}e^{-1}\\ e^{-\frac12}\\ e^{-\frac12}\end{array}\right]\right\|_1  }\\
&\displaystyle{=\left\| \frac{1}{1-e^{-1}}\left[\begin{array}{ccc}
1+e^{-1}&-e^{-\frac12}&-e^{-\frac12}\\
-e^{-\frac12}&1&0\\
-e^{-\frac12}&0&1 \end{array}\right]
\left[\begin{array}{c}
e^{-1}\\ e^{-\frac12}\\
e^{-\frac12}\end{array}
\right]\right\|_1  }\\
&\displaystyle{=\|(-e^{-1}, e^{-\frac12}, e^{-\frac12})^{\top}\|_1 =e^{-1}+2e^{-\frac12}>1.}
\end{array}
$$
In general, for any $n\ge3$, we choose $n+1$ points, $x_1={\bf 0}$, $x_{l+1}=\ve_l/2$ in $\bR^n$ for all $l\in\bN_n$. Here $\ve_l$ is a column vector in $\bR^n$ whose $l$-th component is $1$ and other components are $0$. Let $\bx:=\{x_1,x_2,\dots,x_{n+1}\}$. Then we compute
$$
\begin{array}{ll}
\displaystyle{\supp_{t\in\bR^n}\big\|K[\bx]^{-1}K_\bx(t)\big\|_1}
&\displaystyle{\ge \Big\|K[\bx]^{-1}K_\bx\Big(\frac{\ve_1+\ve_2}{2}\Big)\Big\|_1}\\
&\displaystyle{=\big\|(-e^{-1},e^{-\frac12},e^{-\frac12},0,\dots,0)^{\top}\big\|_1=e^{-1}+2e^{-\frac12}>1.}
\end{array}
$$
In other words, the vector $K_\bx(\frac{\ve_1+\ve_2}{2})$ can be exactly represented by the first three columns of $K[\bx]$. The proof is hence complete.
\end{proof}

\begin{theorem} The Gaussian kernel $K(x,x')=e^{-\|x-x'\|_2^2}$, $x,x'\in\bR^n$ does not satisfy the Lebesgue constant condition \eqref{eq:A4} for any $n\ge 1$.
\end{theorem}
\begin{proof} When $n=1$, we choose two points $x_1=0$ and $x_2=1/2$ in $\bR$.  Let $\bx:=\{x_1,x_2\}$. Then we compute the Lebesgue constant of the Gaussian kernel on $\bR$
$$
\supp_{t\in\bR}\big\|K[\bx]^{-1}K_\bx(t)\big\|_1\ge  \|K[\bx]^{-1}K_\bx(1)\|_1=\|(-e^{-\frac12},e^{-\frac14}+e^{-\frac34})^{\top}\|_1=e^{-\frac12}+e^{-\frac14}+e^{-\frac34}>1.
$$
Generally, for any $n\ge2$, we choose $n+1$ points, $x_1={\bf 0}$, $x_{l+1}=\ve_l/2$ in $\bR^n$ for all $l\in\bN_n$. Let $\bx:=\{x_1,x_2,\dots,x_{n+1}\}$. Then we compute
$$
\begin{array}{ll}
\displaystyle{\supp_{t\in\bR^n}\big\|K[\bx]^{-1}K_\bx(t)\big\|_1}
&\displaystyle{\ge \Big\|K[\bx]^{-1}K_\bx\Big(\frac{\ve_1+\ve_2}{2}\Big)\Big\|_1}\\
&\displaystyle{=\big\|(-e^{-\frac12},e^{-\frac14},e^{-\frac14},0,\dots,0)^{\top}\big\|_1=e^{-\frac12}+2e^{-\frac14}>1.}
\end{array}
$$
In other words, the vector $K_\bx(\frac{\ve_1+\ve_2}{2})$ can be exactly represented by the first three columns of $K[\bx]$. The proof is hence complete.
\end{proof}

We remark that the Lebesgue constant for the kernel interpolation always satisfies
$$
\supp_{t\in X}\big\|K[\bx]^{-1}K_\bx(t)\big\|_1\ge \supp_{t\in \bx}\big\|K[\bx]^{-1}K_\bx(t)\big\|_1 =1.
$$
Therefore, asking it to be exactly bounded below by $1$ is a very strong condition and only a few kernels satisfy it.
To address this problem, we are devoted to investigating a relaxed version of the representer theorem and the Lebesgue constant condition in the next section.

\section{A Relaxed Representer Theorem}\label{sec:relaxed}
As shown in Section \ref{sec:examples}, the Lebesgue constant condition \eqref{eq:A4} is a very strong condition and only a few commonly used kernels satisfy it. We will derive in this section relaxed representer theorems that need a weaker condition on the Lebesgue constant. We will also present a rich class of kernels that satisfy this weaker condition.

In particular, we will consider the {\it relaxed linear representer theorem} in the following form
\begin{equation}\label{remedyrepresenter}
\min_{f\in \cS^{\bx}}L(f(\bx),\by)+\lambda \|f\|_{\cB_\bK}\le \min_{f\in\cB_\bK} L(f(\bx),\by)+\lambda\beta_m \|f\|_{\cB_\bK},
\end{equation}
where $\beta_m\ge1$ is a constant depending on the number $m$ of sampling points, the kernel $\bK$ and the input space $X$.

We point out that if we require $\beta_m=1$ for all $m$, then it is exactly the same as the Lebesgue constant condition \eqref{eq:A4}. Once allowing $\beta_m>1$, there would be a large class of kernels included in our framework. On the other hand side, as long as $\beta_m$ is bounded on $m$ or does not increase too fast with respect to $m$, we will still get a reasonable learning rate of the regularization networks model in machine learning \cite{SongZhang}.

We shall prove a weaker condition on the Lebesgue constant for the relaxed linear representer theorem \eqref{remedyrepresenter}. To this end, we first show a connection between the relaxed linear representer theorem for regularization networks and that for the minimal norm interpolation problem.
\begin{lemma}
If there exists some $\beta_m \geq 1$ such that for all $\vy\in\bR^{md}$
\begin{equation}\label{remedysufficientcond}
\min_{f\in\cI_\bx(\by)}\|f\|_{\cB_\bK}\ge \frac1{\beta_m}\min_{\cI_\bx(\by)\cap \cS^\bx}\|f\|_{\cB_\bK}
\end{equation}
then the relaxed linear representer theorem (\ref{remedyrepresenter}) holds true for any continuous loss function $V$ and any regularization parameter $\lambda$.
\end{lemma}
\begin{proof}
Suppose $f_0$ is a minimizer of $\min_{f\in\cB_\bK}L(f(\bx),\by)+\lambda\beta_m \|f\|_{\cB_\bK}$. Let $g$ be the unique function in $\cS^\bx$ that interpolates $f_0$ at $\bx$, namely, $g(\bx)=f_0(\bx)$. By (\ref{remedysufficientcond}),
$\|g\|_{\cB_\bK}\le \beta_m\|f_0\|_{\cB_\bK}$. It implies
$$
L(g(\bx),\by)+\lambda\|g\|_{\cB_\bK}\le L(f_0(\bx),\by)+\lambda\beta_m\|f_0\|_{\cB_\bK},
$$
which finishes the proof.
\end{proof}

The next result gives a characterization for condition \eqref{remedysufficientcond}. It is a weaker version of the Lebesgue constant condition \eqref{eq:A4}.
\begin{theorem}
Equation (\ref{remedysufficientcond}) holds true for all $\vy\in\bR^{md}$ if and only if
\begin{equation}\label{relaxation}
\sup_{t\in X}\|\bK[\bx]^{-1}\bK_\bx(t)\|_1\le \beta_m.
\end{equation}
\end{theorem}
\begin{proof} Remember that the set $\cI_{\bx}(\by)\cap\cS^{\bx}$ consists of only one function $f_0:=\bK^{\bx}(\cdot)\bK[\bx]^{-1}\by$.
Let $g$ be an arbitrary function in $\cI_\bx(\by)\cap\cB_0$, where $\cB_0$ is defined in (\ref{B0}).
By adding sampling points and assigning the corresponding coefficients to be zero if necessary, we may assume $g\in \cS^{\bx\cup\bt}\cap \cI_\bx(\by)$ for a set of new points $\bt:=\{t_k\in X: k\in\bN_n\}$ disjoint with $\bx$.
Let $\bb:=g(\bt)$, and denote by $\bK[\bt,\bx]$ and $\bK[\bx,\bt]$ the $md\times nd$ and $nd\times md$ matrices given by
$$
(\bK[\bt,\bx])_{jk}:=\bK(t_k,x_j),\ \ j\in\bN_m,k\in\bN_n, \quad \mbox{and}\quad (\bK[\bx,\bt])_{jk}:=\bK(x_k,t_j):\ \ j\in\bN_n,k\in\bN_m.
$$
It then follows
\begin{equation}\label{remedythmeq1}
\|g\|_{\cB_\bK}=\left\|\left[\begin{array}{cc}
\bK[\bx]&\bK[\bt,\bx]\\
\bK[\bx,\bt]&\bK[\bt]
\end{array}\right]^{-1}
\left[\begin{array}{c}
\by\\ \bb
\end{array}\right]\right\|_1=\left\|\left[
\begin{array}{c}
\bK[\bx]^{-1}\by-\bK[\bx]^{-1}\bK[\bt,\bx]\tilde{\bb}\\
\tilde{\bb}
\end{array}\right]\right\|_1,
\end{equation}
where
$$
\tilde{\bb}:=\Big(\bK[\bt]-\bK[\bx,\bt]\bK[\bx]^{-1}\bK[\bt,\bx]\Big)^{-1}(\bb-\bK[\bx,\bt]\bK[\bx]^{-1}\by).
$$
Note that as $\bb$ is allowed to take any vector in $\bR^{nd}$, so is $\tilde{\bb}$.

If (\ref{remedysufficientcond}) holds true for all $\by\in\bR^{md}$ then we choose $\bt$ to be a singleton $\{t_1\}$, $\tilde{\vb}=\ve_j$, and $\by=\bK[t_1,\bx]\ve_j=\bK_\bx(t_1)\ve_j$ for some $j\in\bN_d$. It follows
$$
1=\left\|\left[
\begin{array}{c}
{\bf 0}_{md}\\
\ve_j
\end{array}\right]\right\|_1
\ge\frac1{\beta_m}\|f_0\|_{\cB_\bK}=\frac1{\beta_m}\left\|\bK[\bx]^{-1}\by\right\|_1
=\frac1{\beta_m}\left\|\bK[\bx]^{-1}\bK_\bx(t_1)\ve_j\right\|_1.
$$
As $j\in\bN_d$ is arbitrary,  we get (\ref{relaxation}).

Conversely, suppose that (\ref{relaxation}) is satisfied. We need to show that for all $g\in\cI_\bx(\by)$
$$
\|g\|_{\cB_\bK}\ge \frac{1}{\beta_m}\|f_0\|_{\cB_\bK}=\frac1{\beta_m}\left\|\bK[\bx]^{-1}\by\right\|_1.
$$
We shall discuss the case when $g\in\cI_\bx(\by)\cap\cB_0$ only, as the general case will then follow from the same arguments as those in the last paragraph of the proof of Theorem \ref{thm:repinterp}. Let $g\in\cI_\bx(\by)\cap\cB_0$ with the norm in Equation (\ref{remedythmeq1}). If $\|\bK[\bx]^{-1}\by\|_1\le\beta_m\|\tilde{\bb}\|_1$, it is direct to observe that
$$
\|g\|_{\cB_\bK}\ge \|\tilde{\bb}\|_1\ge \frac1{\beta_m}\left\|\bK[\bx]^{-1}\by\right\|_1.
$$
On the other hand, if $\|\bK[\bx]^{-1}\by\|_1\ge\beta_m\|\tilde{\bb}\|_1$, then by  (\ref{relaxation}) we have
$$
\begin{array}{rl}
\|g\|_{\cB_\bK}&\displaystyle{\ge \|\bK[\bx]^{-1}\by\|_1-\|\bK[\bx]^{-1}\bK[\bt,\bx]\tilde{\bb}\|_1+\|\tilde{\bb}\|_1}\\
&\displaystyle{\ge \|\bK[\bx]^{-1}\by\|_1-\left(\max_{k\in\bN_n}\|\bK[\bx]^{-1}\bK_\bx(t_k)\|_1\right)\|\tilde{\bb}\|_1+\|\tilde{\bb}\|_1}\\
&\displaystyle{\ge \|\bK[\bx]^{-1}\by\|_1-(\beta_m-1)\|\tilde{\bb}\|_1 }\\
&\displaystyle{> \|\bK[\bx]^{-1}\by\|_1-(\beta_m-1)\frac1{\beta_m}\|\bK[\bx]^{-1}\by\|_1}\\
&\displaystyle{=\frac1{\beta_m}\|\bK[\bx]^{-1}\by\|_1}.
\end{array}
$$
The proof is hence complete.
\end{proof}

In the rest of this section, we discuss examples of admissible kernels that satisfy the weaker Lebesgue constant condition \eqref{relaxation}. The Lebesgue constants can measure the stability of kernel-based interpolation. Toward this research interest, it was proved in \cite{Hangelbroek2010}  that the Lebesgue constant for the reproducing kernel of Sobolev space on a compact domain is uniformly bounded for quasi-uniform input points (see, Theorem 4.6 therein). For translation invariant kernels $K(x,x')=\phi(x-x')$, $x,x'\in\bR^n$, the paper \cite{DeMarchi2010} showed that if the Fourier transform $\hat{\phi}$ of $\phi$ satisfies
\begin{equation}\label{Fourierdecay}
c_1(1+\|\xi\|_2^2)^{-\tau}\le \hat{\phi}(\xi)\le c_2(1+\|\xi\|_2^2)^{-\tau},\ \|\xi\|>M
\end{equation}
for some positive constants $c_1,c_2,M$ and $\tau$, the Lebesgue constant for quasi-uniform inputs is bounded by a multiple of $\sqrt{m}$.
This includes, for example, Poisson radial functions, Mat\'{e}rn kernels  and Wendland's compactly supported kernels \cite{DeMarchi2010,Wendland2005}.
In particular, multivariate exponential kernels $e^{-\|x-x'\|_1/r}$, $x,x'\in\bR^n$ satisfy (\ref{Fourierdecay}), where $r>0$.

\section{Numerical Experiments}

We shall perform numerical experiments to show that the regularization network (\ref{MinimizationBK1}) in vector-valued RKBSs (VVRKBS) with the $\ell^1$ norm is indeed able to yield sparsity compared to the one in vector-valued RKHSs (VVRKHS).
Moreover, we can achieve better numerical performance for multi-task learning in the constructed spaces.

Here and subsequently,  the multi-variate exponential kernel takes the form
$$
\bK(x,x'):=K_{\exp}(x,x')\bA=e^{-\|x-x'\|_1/r}\bA,\ x,x'\in \bR^n,
$$
where $r>0$ and $\bA$ denotes a $d\times d$ positive definite symmetric matrix. Let $\cB_{\bK}$ be the associated vector-valued RKBS  with the $\ell^1$ norm and $\cH_{\bK}$ the vector-valued RKHS with reproducing kernel $\bK$. For the sake of simplicity,  the square loss function will be used. We compare the following regularization network models
$$
\min_{f\in\cB_\bK}\|f(\bx)-\by\|_2^2+\lambda \|f\|_{\cB_\bK}
$$
and
$$
\min_{f\in\cB_\bK}\|f(\bx)-\by\|_2^2+\lambda \|f\|^2_{\cH_\bK}.
$$
By the relaxed linear representer theorem for $\cB_\bK$ and the linear representer theorem for $\cH_\bK$, the minimizers of the previous models are
$$
\bK^\bx(\cdot){\bf b}=\bA\sum_{j=1}^m e^{-\|x_j-\cdot\|_1/r}{\bf b}_j \mbox{ with }{\bf b}:=\arg\min_{\bc\in\bR^{md}}\Big\{\|\bK[\bx]\bc-\by\|_2^2+\lambda \|\bc\|_1\Big\}
$$
and
$$
\bK^\bx(\cdot){\bf h}=\bA\sum_{j=1}^m e^{-\|x_j-\cdot\|_1/r}{\bf h}_j\mbox{ with }{\bf h}:=\arg\min_{\bc\in\bR^{md}}\Big\{\|\bK[\bx]\bc-\by\|_2^2+\lambda \bc^{\top} \bK[\bx]\bc\Big\},
$$
respectively.  The $\ell^1$-regularized least square regression problem about ${\bf b}$ does not have a closed form solution. We employ the alternating direction method of multipliers (ADMM) \cite{Boyd2011} to solve it. 
The coefficient vector ${\bf h}$ has the closed form ${\bf h}=(\bK[\bx]+\lambda I_{md})^{-1}\by$.
We run all the experiments on a computer with a single NVIDIA Quadro P2000.

The first numerical experiment is for synthetic data. In this experiment, we set $r:=1$.
The training data is generated by a function $f:\bR^2\to\bR^3$ defined as
$$
f(x):=\bA\big(e^{-\|x-(1,1)\|_1}c_1+e^{-\|x-(0.5,0.5)\|_1}c_2+e^{-\|x\|_1}c_3+e^{-\|x+(0.8,0.8)\|_1}c_4+e^{-\|x+(1,1)\|_1}c_5\big), x\in\bR^2,
$$
where
$$
[c_1,c_2,c_3,c_4,c_5]:=\left[\begin{array}{ccccc}
1 & 1& 1& 1& 1\\
1 & 1/2 & 1 &1/2 &1\\
1/2&1 & 1 &1 & 1/2\\
\end{array}
\right]
\mbox{ and }
\bA:=\left[\begin{array}{ccc}
1&e^{-1}&e^{-2}\\
e^{-1}&1&e^{-1}\\
e^{-2}&e^{-1}&1
\end{array}
\right].
$$
Let $\bx:=\{(0.1 i, 0.1j): -10\le i,j\le 10\}$ be the set of $441$ grid points in $[-1,1]^2$, and use the output vector $\by:=(f(x_1),\dots,f(x_{441}))^{\top}\in\bR^{1323}$ at $\bx$ which are then disturbed by some noise. The regularization parameter $\lambda$ for each model will be optimally chosen from $\{10^j: j=-5,-4,\dots,1\}$ so that the mean square error (MSE) between predicted values and $\by$ will be minimized. We then compare the performance measured by the MSE and the sparsity for two regularization models. The sparsity is measured by the number of nonzero components in the coefficient vectors ${\bf b}$ and ${\bf h}$. We test both models with two types of noise: Gaussian noise with variance $0.01$, and uniform noise in $[-0.1,0.1]$. For each type of noise, we run 50 times of numerical experiments and compute the average MSE, the average sparsity, and the maximum sparsity in the 50 experiments. We conclude that
the regularization network in the vector-valued RKBS with the $\ell^1$ norm outperforms the classical one for synthetic data. At the same time, the sparsity of data representation can be substantially promoted in our constructed space. The results are listed in Table \ref{Tab1}.

\begin{table}[htbp]
\centering
\caption{Comparison of the least square regularization for synthetic data in vector-valued RKHS and in vector-valued RKHS with the $\ell^1$ norm.} \label{Tab1}
\begin{tabular}{|c|cc|cc|} \hline
&Gaussian noise & &Uniform noise&  \\ \hline
&MSE&Sparsity (Max) &MSE& Sparsity (Max) \\ \hline
VVRKHS&0.0032&1323 (1323)   &0.0022&1323 (1323) \\ \hline
VVRKBS&0.0016&47.6 (57) &0.0010& 67.7 (93) \\ \hline
\end{tabular}
\end{table}

The second experiment is for the MNIST database (\url{http://yann.lecun.com/exdb/mnist/}) of handwritten digits from machine learning repository. It possesses a training set of 7291 examples, and a test set of 2007 examples. Each digit is a vector in $[0,1]^{255}$. Limited by the computation resource, we only choose three digits $6,8$, and $9$. Then we have a set $\bx$ of 1850 examples for training, and a set of 513 examples for testing.  For the multi-task learning, labels $6,8$ and $9$ are transferred to the vectors
$(1,0,0)$, $(0,1,0)$, and $(0,0,1)$, respectively. For a set  $\bz\subseteq\bx$ of 60 randomly chosen examples, $\{\|x-x'\|_1: x,x'\in \bz\}$ has mean $72.50$ and standard deviation $20.95$. Therefore, we choose the variance $r:=70$ in the second experiment.

We compute the prediction accuracy for training data and the sparsity of coefficients for both models.
Then we apply learned coefficients of both models to testing data. The accuracy is measured by labels that are correctly predicted by models.
The results are listed in Table \ref{Tab2}. To be more specific, we pick out the digits from the testing data that are misclassified by models.
We number the testing data with numbers from $1$ to $513$. The numbers of $8$ misclassified digits, predicted labels, and true labels for each model are listed in Table \ref{Tab3}. Both regularization models classify the numbers $56$, $58$, $73$, $113$, $212$, $430$, and $480$ incorrectly.
But the number $64$ is misclassified only in VVRKHS and the number $255$ misclassified only in VVRKBS.
The original images of $9$ misclassified digits for both models are displayed in Figure \ref{Fig1}. The numerical performances for both models are comparable.

\begin{table}[htbp]
\centering
\caption{Comparison of the classification for digits $6,8,9$ in VVRKHS and in VVRKHS with the $\ell^1$ norm for the exponential kernel with variance $r=70$.} \label{Tab2}
\begin{tabular}{|c|c|c|c|} \hline
&Accuracy for training data & Sparsity&Accuracy for testing data\\ \hline
VVRKHS& 100\% &5550  &98.44\% \\    
VVRKBS& 100\% & 1455  &98.44\% \\ \hline
\end{tabular}
\end{table}

\begin{table}[htbp]
\centering
\caption{Misclassified digits $6,8,9$ in VVRKHS and in VVRKBS with the $\ell^1$ norm for the exponential kernel with variance $r=70$.} \label{Tab3}
\begin{tabular}{|ccc||ccc|} \hline
&VVRKHS&&&VVRKBS&  \\
Numbers&True labels&Predicted labels&Numbers&True labels& Predicted labels \\ \hline
56&6&8&56&6&8  \\ \hline
58&8&6&58&8&6  \\ \hline
{\bf 64}&8&6&73&9&8  \\ \hline
73&9&8&113&8&9 \\ \hline
113&8&9&212&8&9 \\ \hline
212&8&9&{\bf 255}&8&9 \\ \hline
430&8&9&430&8&9  \\ \hline
480&9&8&480&9&8  \\ \hline
\end{tabular}
\end{table}

\begin{figure}[hbtp]
\centering
\includegraphics[width=6in]{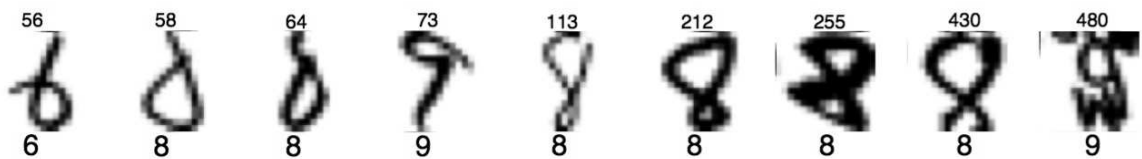}\\
\caption{Misclassified digits $6,8,9$ for both regularization network models.}\label{Fig1}
\end{figure}

To sum up, numerical experiments for both synthetic data and real-world benchmark data have shown us the advantages of multi-task learning in the vector-valued RKBSs with the $\ell^1$ norm.

\end{document}